\documentclass[12pt, reqno]{amsart}

\usepackage{amssymb,latexsym,amsmath,amsfonts}
\usepackage{mathrsfs}
\usepackage{graphicx}
\usepackage[usenames]{color}
\usepackage{hyperref}
\usepackage{comment}
\usepackage{enumitem}

\definecolor{DPurple}{rgb}{0.46,0.2,0.69}

\numberwithin{equation}{section}

\allowdisplaybreaks

\theoremstyle{definition}
\newtheorem{definition}{Definition}[section]

\theoremstyle{remark}

 \theoremstyle{plain}
\newtheorem{theorem}[definition]{Theorem}

\newtheorem{lemma}[definition]{Lemma}



\begin{document}

\title[$\varphi$-Normal Harmonic Mappings]{Lappan's five-point theorem for $\varphi$-Normal Harmonic Mappings}

\author{Nisha Bohra}
\address{Department of Mathematics, Sri Venkateswara College, University of Delhi, Delhi, India
}
\email{ nishabohra@svc.ac.in}
\author{Gopal Datt}
\address{Department of Mathematics, Babasaheb Bhimrao Ambedkar University, Lucknow, India}
\email{ggoapl.datt@gmail.com;   \qquad gopal.du@bbau.ac.in}
\author{Ritesh Pal}
\address{Department of Mathematics, Babasaheb Bhimrao Ambedkar University, Lucknow, India}
\email{rriteshpal@gmail.com}



\begin{abstract}
A harmonic mapping $f=h+\overline{g}$ in $\mathbb{D}$ is $\varphi$-normal if 
$f^{\#}(z)=\mathcal{O}(|\varphi(z)|),  \text{ as } |z|\to 1^-,$ where  $f^{\#}(z)={(|h'(z)|+|g'(z)|)}/{(1+|f(z)|^2)}.$ In this paper,
we establish several sufficient conditions for harmonic mappings to be $\varphi$-normal. We also extend the five-point 
theorem of Lappan for $\varphi$-normal harmonic mappings.
\end{abstract}

\maketitle

\vspace{-0.6cm}
\section{Introduction and Main Results }\label{S:intro}

The notion of {\it normal functions} was introduced by Yosida in \cite{Yosida 34},  and later studied by
Noshiro, although they did not use the term normal function. For the class of normal functions both of them used the 
term ``{\bf class\,(A)}".
The name {\bf normal function} was given by Lehto and Virtanen in their pioneering paper \cite{Lehto 57}, 
wherein they showed that the normal functions 
are closely related  to the problems of boundary behaviour of meromorphic functions.  Let us recall the definition of normal function.  
A  meromorphic function $f$ on the unit 
disc $\mathbb{D}\subset \mathbb{C}$ is normal if the family $\{f\circ\varphi\,:\,\varphi\in \mathsf{Aut}\,(\mathbb{D})\}$ is 
normal in the sense of Montel in 
$\mathbb{D}$, where $\mathsf{Aut}\,(\mathbb{D})$ is the group of conformal automorphisms of the unit disc. 
 Noshiro established a characterizations of normal functions,  (cf . \cite[Theorem~1, p.\,150]{Noshiro 38},
  in terms of the spherical derivative by proving that 
 a meromorphic function is normal iff 
 \begin{equation}\label{eq: nf1}
 \sup_{z\in\mathbb{D}}\,(1-|z|^2)\frac{|f'(z)|}{1+|f(z)|^2}<\infty.
 \end{equation}
 The condition~\eqref{eq: nf1} says that $f$ is Lipschitz when it is considered as a function
from the Poincar\'e's hyperbolic disc $\mathbb{D}$ into the complex projective line $\mathbb{P}^1$ endowed with the
Fubini-Study metric.
Various authors have examined the characteristics of normal meromorphic functions from both geometric and analytical 
perspectives, primarily due to their significance in geometric function theory. 
\smallskip

In this paper, our aim is to extend some of the known results of meromorphic functions to the planar harmonic mappings defined in
$\mathbb{D}$. Let us recall a few things about  planar harmonic mappings. 
Let $D$ be a simply connected domain in $\mathbb{C}$. A harmonic mapping $f$ on $D$ is a complex valued function 
which has the canonical decomposition $f=h+\overline{g}$, where $h$ and $g$ are analytic in $D$ and $g(z_0) =0$ at some 
prescribed point $z_0$ in $D.$ Lewy \cite{lewy} proved that  a harmonic mapping is locally univalent in a domain $D$  if 
and only if its Jacobian $J_f(z)= |h'(z)|^2-|g'(z)|^2$ does not vanish in $D$. If $J_f(z)>0$  throughout the domain $D$, then 
$f$ is a {\it sense-preserving} harmonic map. 
These mappings have been used in study of fluid flows apart from their applications in geometric function theory (cf.
\cite{Aleman 12}).
\smallskip

It is sense to discuss the problem of  complex valued normal 
harmonic mappings, defined in $\mathbb{D}$, given that the topic of harmonic mappings of complex 
value is currently one of the most researched in the field of complex analysis.
Recently, in 2019, H. Arbel\'aez et. al.~\cite{Arbelaz 19} introduced the notion of normal function for 
harmonic mappings. Roughly speaking,
a harmonic mapping $f$ defined in $\mathbb{D}$ is {\bf normal} if it satisfies a Lipschitz type condition.  Arbel\'aez  et. al.   
established a quantitative description of normal harmonic mapping, which we use as the definition.
\begin{definition}\label{D: normal harmonic}\cite[Proposition 2.1]{Arbelaz 19}
A harmonic mapping $f=h+\bar{g}$ in $\mathbb{D}$ is said to be {\it normal} if 
\begin{equation}\label{eq: hnm1}
\sup_{z\in \mathbb{D}} (1-|z|^2)f^{\#}(z)<\infty,
\end{equation}
 where 
 \begin{equation}\label{eq: spherical derivative}
 f^{\#}(z)=\frac{|h'(z)|+|g'(z)|}{1+|f(z)|^2}.
 \end{equation}
 \end{definition}
\smallskip

In this paper, our goal is to examine normal harmonic mappings in a wider context. More precisely,
the purpose is to study the class of those harmonic mappings that satisfy 
\begin{equation*}
f^{\#}(z)=\mathcal{O}(|\varphi(z)|),  \text{ as } 
|z|\to 1^-,
\end{equation*}
 where the growth of the function $\varphi(r)$ surpasses $1/(1-r^2)$, and 
admits a sufficient regularity close to $1$. These classes of harmonic mappings is  larger than the class of normal harmonic mappings, and their
members will be called $\varphi$-normal harmonic mappings.  The analogous class\,--\,for meromorphic functions\,--\,was 
introduced by Aulaskari and R\"{a}tty\"{a} in \cite{Aulaskari 09}. To the best of our knowledge, the class of  
$\varphi$-normal harmonic mappings are not yet explored extensively.  
Numerous beautiful outcomes are achieved in this article, such as the widely recognised rescaling result attributed to 
Lohwater and Pommerenke, which inspired Zalcman to prove the well-known Zalcman's lemma. 
\smallskip

Let us recall the definition of $\varphi$-normal function introduced by Aulaskari and R\"{a}tty\"{a} in \cite{Aulaskari 09}.  
An increasing function $\varphi:[0,1) \rightarrow(0, \infty)$ is called smoothly increasing if

$$
\varphi(r)(1-r) \rightarrow \infty \quad \text { as } r \rightarrow 1^{-}
$$
and
\begin{equation}\label{eq: phi cnvrgnc}
\mathcal{R}_{a}(z):=\frac{\varphi(|a+z / \varphi(|a|)|)}{\varphi(|a|)} \rightarrow 1 \quad \text { as }|a| \rightarrow 1^{-}
\end{equation}
uniformly on compact subsets of $\mathbb{C}$. For a given such $\varphi$, a 
meromorphic function $f $ in $\mathbb{D}$ is called $\varphi$-normal if
$$
\sup _{z \in \mathbb{D}} \frac{f^{\#}(z)}{\varphi(|z|)}<\infty.
$$
Clearly the class of $\varphi$-normal functions in $\mathbb{D}$ is a bigger class than the class of normal functions in $\mathbb{D}$.
\smallskip

Recently, in  \cite{Arbelaz 19, Ponnusamy 20},  authors have defined and discussed various properties of  
normal harmonic mappings. Following the idea of Aulaskari and R\"{a}tty\"{a} (cf. \cite{Aulaskari 09}), 
we define a more general class of $\varphi$-normal harmonic mappings. 

\begin{definition}
A harmonic mapping $f=h+\overline{g}$ in $\mathbb{D}$ is said to be $\varphi$-normal if 
\begin{equation}\label{eq: varphi h normal}
\sup_{z\in\mathbb{D}}\frac{f^{\#} (z)}{\varphi(|z|)}<\infty,
\end{equation}
 where $f^{\#} (z)$ is defined as in \eqref{eq: spherical derivative}.
 \end{definition}

Continuing the study, as in \cite{Arbelaz 19, Ponnusamy 20}, of normal harmonic mapping, we first
extend the rescaling result of Lohwater\,--\,Pommerenke~\cite[Theorem 1]{Lohwater 73}. 
Using the idea of the rescaling result of Lohwater\,--\,Pommerenke, Zalcman proved the famous Zalcman's lemma
in order to make a heuristic principle, namely {\it Robinson\,--\,Zalcman Heuristic Principle}, rigorous (see \cite{Zalcman 75}). The rescaling result 
was studied by many researchers, and extended to higher dimensions as well (see \cite{AladroKrantz, GD 20}).  
For normal harmonic mapping the rescaling result of Lohwater\,--\,Pommerenke 
was studied by Deng et al. in \cite[Theorem 1]{Ponnusamy 20}. We extend this result for $\varphi$-normal harmonic mapping
as follows.

\begin{theorem}\label{T M: Zalcman} 
A non-constant mapping $f$ harmonic in $\mathbb{D}$ is $\varphi$-normal if and only if there do not 
exist sequences $\{z_n\}; \{\rho_n\}$\,---\,where $\rho_n>0$ and $\rho_n\to 0$ as $n\to \infty$\,---\,such that  
$$\lim_{n\to \infty} f\left(z_n+\frac{\rho_n \zeta}{\varphi(|z_n|)}\right)=g(\zeta)$$
locally uniformly in $\mathbb{C}$, where $g$ is a non-constant harmonic mapping.
\end{theorem}
Aulaskari and R\"{a}tty\"{a}, in \cite{Aulaskari 11}, gave a characterization of $\varphi$-normal function in case of 
meromorphic functions (cf. \cite[Theorem~3]{Aulaskari 11}). 
For the case of harmonic mapping, we establish the following 
necessary condition of $\varphi$-normality.

\begin{theorem}\label{T: new Z}
Let $f$ be a harmonic mapping in $\mathbb{D}$ and let $\varphi:[0,1)\rightarrow(0,\infty)$ be an increasing function 
such that $1/\varphi$ is convex and $\lim_{r\rightarrow 1^{-}}\varphi(r)(1-r)=\infty.$ If $f$ is $\varphi$-normal then 
 the family $\{f(z_n+z/\varphi(|z_n|)):n\in \mathbb{Z}_{+}\}$ is normal in $\mathbb{C}$ for any sequence $\{z_n\}_{n=1}^{\infty}$ 
of points in $\mathbb{D}$ such that $\lim_{n\rightarrow\infty} |z_n|=1$. 
\end{theorem}

Answering a question posed by Pommerenke
\cite[Problem 3.2]{Pommerenke 72},
Lappan used the rescaling result \cite[Theorem~1]{Lohwater 73} to establish the well known Lappan's five-point theorem
\cite[Theorem~1]{Lappan 74}. Loosely speaking, the five-point theorem of Lappan says that a meromorphic function 
$f$ is normal in $\mathbb{D}$ if $\sup_{z\in f^{-1}(E)}(1-|z|^2)f^\#(z)$ is bounded for some set $E$ of five points. Datt,
in \cite{GD 23}, gave and extension of Lappan's five-point theorem for the class of $\varphi$ normal functions of several
variables. In this paper, we establish the following 
extension of Lappan's five-point theorem for the $\varphi$-normal harmonic mappings.
\begin{theorem}\label{T: five valued}
Let E be a set of five numbers in $\mathbb{C}$. If $f$ is a sense preserving harmonic 
mapping in $\mathbb{D}$ such that 
\begin{equation}\label{eq: five point cond}
\sup_{z\in f^{-1}(E)}\frac{f^{\#}(z)}{\varphi(|z|)}<\infty.
\end{equation}
Then $f$ is a $\varphi$-normal harmonic function.
\end{theorem}

Lappan commented on the sharpness of the number {\it five} in the five-point theorem \cite[Theorem~1]{Lappan 74}.
He observed that the number {\it five} cannot be replaced by the number {\it three} and in some cases {\it five}
cannot be replaced by {\it four} (see \cite[Theorems~3 and 4]{Lappan 74}). Tan and Thin, 
in \cite[Theorem~1]{Tan 17}, relaxed the hypotheses 
of the \cite[Theorem~1]{Lappan 74} and able to achieve the same conclusion by reducing the number {\it five} to {\it four}. 
In the following theorem we establish a version of the four-point theorem for $\varphi$-normal harmonic mappings.

\begin{theorem}\label{T: four point}
Let $\varphi:[0,1]\to (0,\infty)$ be smoothly increasing, and $f=h+\bar{g}$ be a 
sense-preserving harmonic mapping on $\mathbb{D}$. Suppose that 
$E\subset {\mathbb{C}}$ be a subset of four distinct numbers such that 
\begin{equation*}
\sup_{z\in f^{-1}(E)}\frac{1}{\varphi(|z|)}\left(\frac{|h'(z)|+|g'(z)|}{1+|f(z)|^2}\right)<\infty \ \text{ and }\ 
\sup_{z\in f^{-1}(E)}\left(\frac{|h''(z)|+|g''(z)|}{1+(|h'(z)|+|g'(z)|)^2}\right)<\infty.
\end{equation*}
Then $f$ is $\varphi$-normal harmonic mapping.
\end{theorem}

\section{Essential Lemmas}

In order to prove our results,  we need some essential lemmas. In this section, we shall recall some known results and
prove a few new auxiliary lemmas.
In the theory of normal families, Marty's theorem~\cite{Marty 31} characterizes normality of a family of meromorphic functions by a 
condition wherein the spherical derivative is locally bounded. 
In case of  a family of harmonic mappings, the following sufficient condition  proved in \cite{Ponnusamy 20}
serves as a ``partial"  generalization of Marty's Theorem.

\begin{lemma}\cite{Ponnusamy 20}
A family $\mathcal{F}$ of harmonic mappings $f=h+\overline{g}$ in $\mathbb{D}$ is normal if 
$\{f^{\#}(z)\,:\,f\in \mathcal{F}\}$ is uniformly locally bounded.  
\end{lemma}
We observe that the converse holds true if $\Re\,{h(z)g(z)}\geq0$ in $\mathbb{D}$, where $\Re\,z$
represents the real part of $z$.

\begin{proof}[Proof of the converse part]
We prove the partial converse here. Let $\mathcal{F}$
be a normal family of harmonic mappings $f=h+\overline{g}$ in $\mathbb{D}$. Then we have two families 
$\mathcal{H}_\mathcal{F}:=\{h\,:\,f=h+\overline{g}\}$
and  $\mathcal{G}_\mathcal{F}:=\{g\,:\,f=h+\overline{g}\}$ of holomorphic functions in $\mathcal{D}$. By Marty's theorem we have 
$\{h^\#(z): h\in \mathcal{H}_\mathcal{F}\}$,  and $\{g^\#(z): g\in \mathcal{G}_\mathcal{F}\}$ are uniformly locally bounded.
Therefore, we have 
\begin{align*}
f^\#(z)&=\frac{|h'(z)|+|g'(z)|}{1+|f(z)|^2}=\frac{|h'(z)|}{1+|f(z)|^2}+\frac{|g'(z)|}{1+|f(z)|^2}\\
&\leq \frac{|h'(z)|+|g'(z)|}{1+|h(z)|^2} +\frac{|h'(z)|+|g'(z)|}{1+|g(z)|^2}=h^\#(z)+g^\#(z). 
\end{align*}
The above inequality holds because 
$|f(z)|^2\geq |h(z)|^2+|g(z)|^2$, as $\Re\,h(z)g(z)\geq0$. Now the result follows.
\end{proof}
The next lemma is a generalization of Hurwitz's theorem for harmonic mappings.
\begin{lemma}\cite[Para. 3, P. 10]{Duren}\label{L: Hurwitz}
Let $\{f_n\}$ be a sequence of sense-preserving harmonic mappings in $\mathbb{D}$ such that $f_n$ converges
locally uniformly to a sense-preserving harmonic mapping $f$. Then  $z_0\in\mathbb{D}$ is a zero of $f$ if and 
only if $z_0$ is a cluster point of the zeros of $f_n$, $n\geq 1$.
\end{lemma}
Using lemma~\ref{L: Hurwitz}, Deng et. al. proved the following lemma, which will be used in the proof of our main theorems.

\begin{lemma}\cite[Lemma 3]{Ponnusamy 20}\label{L: multiple sol}
Let $f=h+\overline{g}$ be a sense-preserving harmonic mapping in $\mathbb{C}$ with $g(0)=0$. Then there are  
at most four values of $a$ for which all zeros of $f-a$ are multiple. 
\end{lemma}

We will use the following lemma in order to prove a result similar to Lemma~\ref{L: multiple sol}. We mention here that
researchers working in value distribution theory know the following lemma, nevertheless, for completeness of the paper, we give a
proof of it.
\begin{lemma}\label{L: f z hol}
Let $f$ be a non-constant entire function. Then there are at most three values $a\in \mathbb{C}$ for which
each zero of $f-a$ is of multiplicity at least $3$.
\end{lemma}
\begin{proof}
To prove this lemma, we use the second main theorem of Nevanlinna theory. We assume that the reader 
is acquainted with the terminologies and notations of the Nevanlinna theory. Assume on the contrary that there 
is a set $E=\{a_1, a_2, a_3, a_4\}\subset \mathbb{C}$ such that each zero of $f-a_i$, $a_i\in E$, is of multiplicity 
at least three.Then by the the second main theorem, we get
\begin{align*}
2T(r,f)&\leq \sum_{i=1}^{4}\overline{N}\left(r, \frac{1}{f-a_i}\right)+o(T(r, f))\\
&\leq \frac{4}{3}T(r, f)+o(T(r, f)),
\end{align*}
for all $r\in [1, \infty)$ excluding a set of finite Lebesgue measure. This is a contradiction.
\end{proof}
Using Lemma~\ref{L: f z hol}, and the same technique as in \cite[Lemma 3]{Ponnusamy 20}, we prove the following lemma.
\begin{lemma}\label{L: four zeros}
Let $f$ be a sense preserving harmonic function in $\mathbb{C}$.
Then there are at most three values $a$ for which all zeros of $f-a$ are of multiplicity at least $3$.
\end{lemma}

\begin{proof}
Let $f=h+\overline{g}$ be a sense preserving harmonic mapping in $\mathbb{C}$, and $\omega(z)=g'(z)/h'(z)$. Then
$|\omega(z)|<1$ in $\mathbb{C}$, therefore, by Liouville's theorem, $\omega(z)=\alpha$ with the constant
 $|\alpha|<1$.
As in the proof of Lemma~3 of \cite[p.~611]{Ponnusamy 20}, for any complex number $a$, $f(z)=a$ is 
equivalent to $$h(z)=\frac{a-\overline{\alpha a}+\overline{\alpha h(0)}-|\alpha|^2h(0)}{1-|\alpha|^2}.$$
Now by Lemma~\ref{L: f z hol},  there are at most three values $a$ for which all zeros of $h-a$ are of multiplicity 
at least $3$. Hence, there are  at most three values $a$ for which all zeros of $f-a$ are of multiplicity 
at least $3$.
\end{proof}
\section{Proof of Main Theorems}

\begin{proof}[Proof of Theorem~\ref{T M: Zalcman}]
Suppose $f$ is not $\varphi$-normal, then there exists a sequence $\{z_n^*\}$ of points in $\mathbb{D}$ such that
$|z_n^*|\to 1$ as $n\to \infty$ and 
\begin{equation}\label{Eq: Z1}
\frac{f^{\#}(z_n^*)}{\varphi(|z_n^*|)}\to \infty \qquad\qquad \text{ as } n\to \infty.
\end{equation}
Let $\{r_n\}$ be a sequence such that $|z_n^*|<r_n<1$. Define 
\begin{equation}\label{Eq:Z2}
M_n:=\sup_{|z|\leq r_n}\frac{f^{\#}(z)}{\varphi(|z|)}=\frac{f^{\#}(z_n)}{\varphi(|z_n|)}.
\end{equation}
Define 
\begin{equation}\label{Eq: rho}
\displaystyle\rho_n=1/M_n=\varphi(|z_n|)/f^{\#}(z_n)\to 0  \ \text{ as \ } n\to \infty.
\end{equation} 
Let 
$\displaystyle g_n(\zeta)=f\left(z_n+\frac{\rho_n\zeta}{\varphi(|z_n|)}\right)$ for 
$\displaystyle |\zeta|<R_n=\frac{(1-|z_n|)\varphi(|z_n|)}{\rho_n}\to\infty$ as $n\to \infty$.  Then 
\begin{equation*}
g_n^\#(0) =\frac{\rho_n}{\varphi(|z_n|)}f^{\#}(z_n)=1.
\end{equation*}
We now aim to show that the sequence $\{g_n(\zeta)\}$ is normal. To see this, consider 
\begin{align*}
g_n^\#(\zeta) &= \frac{\rho_n}{\varphi(|z_n|)}f^\#\left(z_n+\frac{\rho_n \zeta}{\varphi(|z_n|)}\right)\\
&= \frac{\rho_n\varphi \left(\big|z_n+\frac{\rho_n \zeta}{\varphi(|z_n|)}\big|\right)}{\varphi(|z_n|)}
\frac{f^\#\left(z_n+\frac{\rho_n \zeta}{\varphi(|z_n|)}\right)}{\varphi \left(\big|z_n+\frac{\rho_n \zeta}{\varphi(|z_n|)}\big|\right)}
\\&= \frac{\varphi \left(\big|z_n+\frac{\rho_n \zeta}{\varphi(|z_n|)}\right)\big|}{\varphi(|z_n|)}
\bigg(\rho_n\frac{f^\#\left(z_n+\frac{\rho_n \zeta}{\varphi(|z_n|)}\right)}{\varphi \left(\big|z_n+\frac{\rho_n \zeta}{\varphi(|z_n|)}\big|\right)
}\bigg)
\rightarrow 1, \ 
\text{ by  \eqref{Eq: rho}}. 
\end{align*}
This shows that the sequence $\{g_n(\zeta)\}$ is normal and it has a convergent subsequence which converges to a harmonic 
mapping $g$ in $\mathbb{C}$ with $g^{\#}(0)=1\neq 0$. Hence, we infer that $g$ a non-constant harmonic mapping.
\smallskip

 We now prove the converse part of the theorem. Assume that $f$ is $\varphi$-normal. By definition, there exists a constant 
 $M>0$ such that $\displaystyle\sup_{z\in D}\frac{f^{\#}(z)}{\varphi(|z|)}\leq M$. The functions $g_n$ defined by 
 $$g_n(\zeta)=f\left(z_n+\frac{\rho_n\zeta}{\varphi(|z_n|)}\right)$$ are defined for 
 $\displaystyle |\zeta|<\frac{1}{\rho_n}(1-|z_n|)\varphi(|z_n|)$.   By the definition of the smoothly increasing function, we get
  $\displaystyle \frac{\rho_n}{(1-|z_n|)\varphi(|z_n|)}\to 0$ as $n\to \infty$.
 Since 
 \begin{align*}
 g_n^{\#}(\zeta)&=\frac{\rho_n}{\varphi(|z_n|)} f^{\#}\left(z_n+\frac{\rho_n\zeta}{\varphi(|z_n|)}\right)\\
&=\frac{\rho_n}{\varphi(|z_n|)} \varphi\left(\big| z_n+\frac{\rho_n\zeta}{\varphi(|z_n|)}\big|\right)
\frac{ f^{\#}\left(z_n+\frac{\rho_n\zeta}{\varphi(|z_n|)}\right)}{\varphi\left(\big| z_n+\frac{\rho_n\zeta}{\varphi(|z_n|)}\big|\right)}\\
&\leq \rho_n \frac{{\varphi\left(\big| z_n+\frac{\rho_n\zeta}{\varphi(|z_n|)}\big|\right)}}{\varphi(|z_n|)} M \to 0 \ \text{ as } n\to \infty.
 \end{align*}
 Thus, $\{g_n\}$ is normal and $g_n^{\#}(\zeta)\to 0$, whence we infer that $g$ is constant, which is a contradiction.
\end{proof}

\begin{proof}[Proof of Theorem~\ref{T: new Z}]
Assume that $f$ is $\varphi$-normal. Let $\{z_n\}\subset \mathbb{D}$ such that $|z_n|\to 1$ as $n\to \infty$. Let 
$z\in D(0, r):=\{z\,:\,|z|<r\}$, and  define 
$\phi_a(z):=a+z/\varphi(|a|)$, for $a\in \mathbb{D}$. Let $f=h+\overline{g}$, since $f$ is harmonic and $\phi_a$ is holomorphic
 it follows that $f\circ \phi_a$ is harmonic.   Therefore,
 \begin{align*}
 (f\circ \phi_{z_n})^\#(z)&=\frac{|(h\circ \phi_{z_n})'(z)|+|(g\circ \phi_{z_n})'(z)|}{1+|f(\phi_{z_n}(z))|^2}\\
 &= \frac{|h'(\phi_{z_n}).\phi_{z_n}'(z)|+|g'(\phi_{z_n}).\phi_{z_n}'(z)|}{1+|f(\phi_{z_n}(z))|^2}\\
 &=\frac{1}{\varphi(|z_n|)}f^{\#}(\phi_{z_n}(z)).
 \end{align*}
 Since $f$ is $\varphi$-normal, i.e., $f^{\#}(z)=O(\varphi(|z|))$, therefore  there exists a positive constant $M$ and an
  $N\in \mathbb{Z}_+$ such that
 \begin{equation*}
(f\circ \phi_{z_n})^\#(z)\leq \frac{1}{\varphi(|z_n|)}M\varphi(|\phi_{z_n}(z)|) \qquad \forall \ n\geq N \ \text{ and } z\in D(0, r).
\end{equation*}
Let $\psi:=1/ \varphi$, then $\psi$ is decreasing and convex, and
\begin{align*}
\lim_{n\to \infty} \sup_{|z|\leq r} \frac{\varphi(|\phi_{z_n}(z)|)}{\varphi(|z_n|)}
&\leq \lim_{n\to \infty}\frac{\psi(|z_n|)}{\psi(|z_n+r\psi(|z_n|))}\leq \lim_{n\to \infty}\frac{1}{1+\psi'(|z_n|)r}
\end{align*}
Since $\displaystyle \lim_{s\to 1^-}\varphi(s)(1-s)=\infty$ we get that $\displaystyle\lim_{s\to 1^-}\frac{1}{1+r\psi'(s)}=1$.
Hence we have $(f\circ \phi_{z_n})^\#(z)\leq 1$ in $D(0, r)$ for all $n\geq N$. Therefore we conclude, by Marty's theorem,
that $\{f\circ \phi_{z_n}\,:\,n\in\mathbb{Z}_+\}$ is normal  in $\mathbb{C}$.
 \end{proof}


\begin{proof}[Proof of Theorem~\ref{T: five valued}]
Assume that $f$ satisfies \eqref{eq: five point cond} for some $E$, with cardinality  $\#E=5$, and $f$ is not $\varphi$-normal. Then
by the rescaling result~\ref{T M: Zalcman}, there exist $\{z_n\}, \ \{\rho_n\}$, $\rho_n>0$ such that    
$\displaystyle\frac{\rho_n}{(1-|z_n|)\varphi(|z_n|)}\to 0,$ and a sense preserving harmonic mapping $g$ such that 
$f\left(z_n+\frac{\rho_n}{\varphi(|z_n|)}\zeta\right)$ converges compactly to $g(\zeta)$. Let $\alpha\in {\mathbb{C}}$ such that 
$g(\zeta)=\alpha$ has a solution $z_0$ which is not a multiple solution. By Lemma~\ref{L: Hurwitz}, for large values of $n$, 
$g_n(\zeta):=f(z_n+\frac{\rho_n}{\varphi(|z_n|)}\zeta)-\alpha$ 
has zeros $\zeta_n$ that converges to $z_0$ as $n\to \infty$. Therefore, $g_n(\zeta_n)-\alpha=0$, and
$$g_n^\#(\zeta_n)=\frac{\rho_n}{\varphi(|z_n|)}f^\#\left(z_n+\frac{\rho_n}{\varphi(|z_n|)}\zeta_n\right).$$
Let $s_n=z_n+\frac{\rho_n}{\varphi(|z_n|)}\zeta_n$. Then
\begin{align*}
\frac{f^\#(s_n)}{\varphi(|s_n|)} &=\frac{\varphi(|z_n|)}{\rho_n}\frac{g_n^\#(\zeta_n)}{\varphi(|s_n|)} \to \infty,
\end{align*}
because, as $n\to \infty$, $g_n^\#(\zeta_n)\to g(\zeta_0)$, ${|s_n|}/{|z_n|}\to 1$, and $\rho_n\to 0$.
\smallskip

We just proved that if $g(\zeta)-\alpha$ has a simple zero, then 
\begin{equation}
\sup_{z\in f^{-1}(E)}\frac{f^{\#}(z)}{\varphi(|z|)}=\infty.
\end{equation}
Lemma~\ref{L: multiple sol},  however, states that there can be at most four values of $\alpha$ 
for which all zeros of $g(\zeta)-\alpha$ are multiple zeros. Consequently, we conclude that if $f$ is a sense-preserving
non-normal harmonic mapping in $\mathbb{D}$, then for each complex number $\alpha$, with a maximum of four exceptions, 
we have
\begin{equation}
\sup_{z\in f^{-1}(E)}\frac{f^{\#}(z)}{\varphi(|z|)}<\infty.
\end{equation}
This completes the proof.
\end{proof}


\begin{proof}[Proof of Theorem~\ref{T: four point}]
Assume on the contrary that $f=h+\bar{g}$ is not $\varphi$-normal, then there exists $\{z_n\}\subset \mathbb{D}$, 
$z_n\to 1^-$, $\rho_n\to 0$ such that the sequence 
\begin{equation}\label{eq: cnvg 4 pt} 
F_n(\zeta)=f\left(z_n+\frac{\rho_n\zeta}{\varphi(|z_n|)}\right)=h\left(z_n+\frac{\rho_n\zeta}{\varphi(|z_n|)}\right)+
\overline{g\left(z_n+\frac{\rho_n\zeta}{\varphi(|z_n|)}\right)}\to F
\end{equation}
converges uniformly on compact subsets to a non-constant harmonic mapping $F(\zeta)=H(\zeta)+\overline{G(\zeta)}$ in 
$\mathbb{C}$. 
Let $$ F_n(\zeta)=h_n(\zeta)+\overline{g_n(\zeta)}, \text{ where } h_n(\zeta):= h\left(z_n+\frac{\rho_n\zeta}{\varphi(|z_n|)}\right) \text{ and }
g_n(\zeta):= g\left(z_n+\frac{\rho_n\zeta}{\varphi(|z_n|)}\right).$$ 
Then $h_n'(\zeta)\to H'(\zeta)$ and $g_n'(\zeta)\to G'(\zeta)$, that is, 
$F_n'(\zeta)=h_n'(\zeta)+\overline{g_n'(\zeta})\to H'(\zeta)+\overline{G'(\zeta)}.$ 
\smallskip

{\bf{ Claim:}} For any number $a\in E$, each zero of $F-a$ has multiplicity at least $3$..
\smallskip

\noindent For any zero of $F-a$, say $\zeta_0$, by Lemma~\ref{L: Hurwitz}, there exists a sequence 
$\{\zeta_n\}$ of points converging to $\zeta_0$ such that 
\begin{equation*}
F_n(\zeta_n)=f\left(z_n+\frac{\rho_n\zeta_n}{\varphi(|z_n|)}\right)=a
\end{equation*}
Then, by the hypotheses, there exists an $M>0$ such that 
\begin{equation*}
\frac{\big|h'\left(z_n+\frac{\rho_n\zeta_n}{\varphi(|z_n|)}\right)\big|+\big|g'\left(z_n+\frac{\rho_n\zeta_n}{\varphi(|z_n|)}\right)\big|}{1
+|f\left(z_n+\frac{\rho_n\zeta_n}{\varphi(|z_n|)}\right)|^2}\leq M\varphi\left(\big|z_n+\frac{\rho_n\zeta_n}{\varphi(|z_n|)}\big|\right),
\end{equation*}
and
\begin{equation*}
\frac{\big|h''\left(z_n+\frac{\rho_n\zeta_n}{\varphi(|z_n|)}\right)\big|+\big|g''\left(z_n+\frac{\rho_n\zeta_n}{\varphi(|z_n|)}\right)\big|}{1
+\left(\big|h'\left(z_n+\frac{\rho_n\zeta_n}{\varphi(|z_n|)}\right)\big|+\big|g'\left(z_n+\frac{\rho_n\zeta_n}{\varphi(|z_n|)}\right)\big|
\right)^2}\leq M.
\end{equation*}

Therefore,
\begin{align}
F_n^\#(\zeta_n)= \frac{|h_n'(\zeta_n)|+|g_n'(\zeta_n)|}{1
+|f_n(\zeta_n)|^2}& =\frac{\rho_n}{\varphi(|z_n|)}\frac{\big|h'\left(z_n+\frac{\rho_n\zeta_n}{\varphi(|z_n|)}\right)\big|+\big|g'\left(z_n+\frac{\rho_n\zeta_n}{\varphi(|z_n|)}\right)\big|}{1
+|f\left(z_n+\frac{\rho_n\zeta_n}{\varphi(|z_n|)}\right)|^2}\notag\\
&\leq \rho_n M\frac{\varphi\left(\big|z_n+\frac{\rho_n\zeta_n}{\varphi(|z_n|)}\big|\right)}{\varphi(|z_n|)}.
\end{align}
From this, \eqref{eq: phi cnvrgnc}, and \eqref{eq: cnvg 4 pt} we get 
\begin{equation*}
\frac{|H'(\zeta_0)|+|G'(\zeta_0)|}{1+|F(\zeta_0)|^2}=0.
\end{equation*}
Since $F$ is non-constant, whence $|H'(\zeta_0)|+|G'(\zeta_0)|=0$. Therefore, we get that $H'(\zeta_0)=0$ and
$G'(\zeta_0)=0$. Hence $\zeta_0$ is a zero of order at least $2$ of $F-a$. We now aim to show that zeros of $F-a$ are of 
order at most $3$.
\smallskip

Assume that $a\in\mathbb{C}$. Then we have
  
 \begin{align*}
\big|H'\left(z_n+\frac{\rho_n\zeta_n}{\varphi(|z_n|)}\right)\big|+\big|G'\left(z_n+\frac{\rho_n\zeta_n}{\varphi(|z_n|)}\right)\big|
&=\left(1+\big|F\left(z_n+\frac{\rho_n\zeta_n}{\varphi(|z_n|)}\right)\big|^2\right)F^\#\left(z_n+\frac{\rho_n\zeta_n}{\varphi(|z_n|)}\right)\\
&\leq M\left(1+\max_{b\in E\setminus\{\infty\}}|b|^2\right)\varphi\left(\big|z_n+\frac{\rho_n\zeta_n}{\varphi(|z_n|)}\big|\right)
\end{align*}
Which further gives

\begin{align}
\frac{|h_n''(\zeta_n)|+|g_n''(\zeta_n)|}{1+(|h_n'(\zeta_n)|+|g_n'(\zeta_n)|)^2}
&=\frac{\rho_n^2}{\varphi^2(|z_n|)}\frac{\big|h''\left(z_n+\frac{\rho_n\zeta_n}{\varphi(|z_n|)}\right)\big|
+\big|g''\left(z_n+\frac{\rho_n\zeta_n}{\varphi(|z_n|)}\right)\big|}{1+\frac{\rho_n^2}{\varphi^2(|z_n|)}\left(\big|h'\left(z_n+
\frac{\rho_n\zeta_n}{\varphi(|z_n|)}\right)\big|
+\big|g'\left(z_n+\frac{\rho_n\zeta_n}{\varphi(|z_n|)}\right)\big|
\right)^2}\notag\\
&= \frac{\rho_n^2}{\varphi^2(|z_n|)}\frac{\big|h''\left(z_n+\frac{\rho_n\zeta_n}{\varphi(|z_n|)}\right)\big|
+\big|g''\left(z_n+\frac{\rho_n\zeta_n}{\varphi(|z_n|)}\right)\big|}{1+\left(\big|h'\left(z_n+\frac{\rho_n\zeta_n}{\varphi(|z_n|)}\right)\big|
+\big|g'\left(z_n+\frac{\rho_n\zeta_n}{\varphi(|z_n|)}\right)\big|\right)^2}\times\notag\\
& \qquad \qquad  \frac{1+\left(\big|h'\left(z_n+\frac{\rho_n\zeta_n}{\varphi(|z_n|)}\right)\big|
+\big|g'\left(z_n+\frac{\rho_n\zeta_n}{\varphi(|z_n|)}\right)\big|\right)^2}{1+\frac{\rho_n^2}{\varphi^2(|z_n|)}\left(\big|h'\left(z_n+
\frac{\rho_n\zeta_n}{\varphi(|z_n|)}\right)\big|
+\big|g'\left(z_n+\frac{\rho_n\zeta_n}{\varphi(|z_n|)}\right)\big|
\right)^2}\notag\\
&\leq \frac{\rho_n^2}{\varphi^2(|z_n|)} M \left(1+\left(\big|h'\left(z_n+\frac{\rho_n\zeta_n}{\varphi(|z_n|)}\right)\big|
+\big|g'\left(z_n+\frac{\rho_n\zeta_n}{\varphi(|z_n|)}\right)\big|\right)^2\right)\label{eq: Double}
\end{align}
Since 
\begin{equation*}
\frac{\big|h'\left(z_n+\frac{\rho_n\zeta_n}{\varphi(|z_n|)}\right)\big|
+\big|g'\left(z_n+\frac{\rho_n\zeta_n}{\varphi(|z_n|)}\right)\big|}{1+|f\left(z_n+\frac{\rho_n\zeta_n}{\varphi(|z_n|)}\right)|^2}
\leq M\varphi\left(\big|z_n+\frac{\rho_n\zeta_n}{\varphi(|z_n|)}\big|\right)
\end{equation*}
we get
\begin{equation*}
\left(\big|h'\left(z_n+\frac{\rho_n\zeta_n}{\varphi(|z_n|)}\right)\big|
+\big|g'\left(z_n+\frac{\rho_n\zeta_n}{\varphi(|z_n|)}\right)\big|\right)^2
\leq M^2\varphi^2\left(\big|z_n+\frac{\rho_n\zeta_n}{\varphi(|z_n|)}\big|\right)
\left({1+|f\left(z_n+\frac{\rho_n\zeta_n}{\varphi(|z_n|)}\right)|^2}\right)^2.
\end{equation*}
Using this in \eqref{eq: Double}, we get
\begin{align}
\frac{|h_n''(\zeta_n)|+|g_n''(\zeta_n)|}{1+(|h_n'(\zeta_n)|+|g_n'(\zeta_n)|)^2}
&\leq \frac{\rho_n^2}{\varphi^2(|z_n|)} M \left(1+M^2\left(1+\max_{b\in E\setminus\{\infty\}}|b|^2\right)^2. 
\varphi^2\left(\big|z_n+\frac{\rho_n\zeta_n}{\varphi(|z_n|)}\big|\right)\right)\notag\\
&\leq \rho_n^2 M \left(1+M^2\left(1+\max_{b\in E\setminus\{\infty\}}|b|^2\right)^2. 
\left(\frac{\varphi\left(\big|z_n+\frac{\rho_n\zeta_n}{\varphi(|z_n|)}\big|\right)}{\varphi(|z_n|}\right)^2\right)\notag\\
&\longrightarrow 0.\notag
\end{align}
Thus, we get that $|H''(\zeta_0)|+|G''(\zeta_0)|=0$, whence, $H''(\zeta_0)=G''(\zeta_0)=0$. This implies that 
$\zeta_0$ is a zero, of multiplicity at least 3, of $F-a$.  Since cardinality of $E$ is $4$ we get a contradiction 
by Lemma~\ref{L: four zeros}. Thus, $f$ is $\varphi$-normal harmonic mapping.
\end{proof}

\medskip

		\noindent{\bf Data Availability.} There is no additional data involved.

\end{document}